\documentclass[12pt]{amsart}
\usepackage{graphics}
\usepackage{amsfonts,amssymb,color}
\usepackage[mathscr]{eucal}
\usepackage{amsmath, amsthm}
\usepackage{mathrsfs}
\usepackage{amsbsy}
\usepackage{dsfont}
\usepackage{bbm}
\usepackage{wasysym}
\usepackage{stmaryrd}
\usepackage{url}

\input xypic
\xyoption{all}

\makeindex

\begin{document}
\baselineskip = 16pt

\newcommand \ZZ {{\mathbb Z}}
\newcommand \NN {{\mathbb N}}
\newcommand \RR {{\mathbb R}}
\newcommand \PR {{\mathbb P}}
\newcommand \AF {{\mathbb A}}
\newcommand \QQ{{\mathbb Q}}
\newcommand \GG {{\mathbb G}}
\newcommand \bcA {{\mathscr A}}
\newcommand \bcC {{\mathscr C}}
\newcommand \bcD {{\mathscr D}}
\newcommand \bcF {{\mathscr F}}
\newcommand \bcG {{\mathscr G}}
\newcommand \bcH {{\mathscr H}}
\newcommand \bcM {{\mathscr M}}
\newcommand \bcJ {{\mathscr J}}
\newcommand \bcL {{\mathscr L}}
\newcommand \bcO {{\mathscr O}}
\newcommand \bcP {{\mathscr P}}
\newcommand \bcQ {{\mathscr Q}}
\newcommand \bcR {{\mathscr R}}
\newcommand \bcS {{\mathscr S}}
\newcommand \bcV {{\mathscr V}}
\newcommand \bcW {{\mathscr W}}
\newcommand \bcX {{\mathscr X}}
\newcommand \bcY {{\mathscr Y}}
\newcommand \bcZ {{\mathscr Z}}
\newcommand \goa {{\mathfrak a}}
\newcommand \gob {{\mathfrak b}}
\newcommand \goc {{\mathfrak c}}
\newcommand \gom {{\mathfrak m}}
\newcommand \gon {{\mathfrak n}}
\newcommand \gop {{\mathfrak p}}
\newcommand \goq {{\mathfrak q}}
\newcommand \goQ {{\mathfrak Q}}
\newcommand \goP {{\mathfrak P}}
\newcommand \goM {{\mathfrak M}}
\newcommand \goN {{\mathfrak N}}
\newcommand \uno {{\mathbbm 1}}
\newcommand \Le {{\mathbbm L}}
\newcommand \Spec {{\rm {Spec}}}
\newcommand \Gr {{\rm {Gr}}}
\newcommand \Pic {{\rm {Pic}}}
\newcommand \Jac {{{J}}}
\newcommand \Alb {{\rm {Alb}}}
\newcommand \Corr {{Corr}}
\newcommand \Chow {{\mathscr C}}
\newcommand \Sym {{\rm {Sym}}}
\newcommand \Prym {{\rm {Prym}}}
\newcommand \cha {{\rm {char}}}
\newcommand \eff {{\rm {eff}}}
\newcommand \tr {{\rm {tr}}}
\newcommand \Tr {{\rm {Tr}}}
\newcommand \pr {{\rm {pr}}}
\newcommand \ev {{\it {ev}}}
\newcommand \cl {{\rm {cl}}}
\newcommand \interior {{\rm {Int}}}
\newcommand \sep {{\rm {sep}}}
\newcommand \td {{\rm {tdeg}}}
\newcommand \alg {{\rm {alg}}}
\newcommand \im {{\rm im}}
\newcommand \gr {{\rm {gr}}}
\newcommand \op {{\rm op}}
\newcommand \Hom {{\rm Hom}}
\newcommand \Hilb {{\rm Hilb}}
\newcommand \Sch {{\mathscr S\! }{\it ch}}
\newcommand \cHilb {{\mathscr H\! }{\it ilb}}
\newcommand \cHom {{\mathscr H\! }{\it om}}
\newcommand \colim {{{\rm colim}\, }} 
\newcommand \End {{\rm {End}}}
\newcommand \coker {{\rm {coker}}}
\newcommand \id {{\rm {id}}}
\newcommand \van {{\rm {van}}}
\newcommand \spc {{\rm {sp}}}
\newcommand \Ob {{\rm Ob}}
\newcommand \Aut {{\rm Aut}}
\newcommand \cor {{\rm {cor}}}
\newcommand \Cor {{\it {Corr}}}
\newcommand \res {{\rm {res}}}
\newcommand \red {{\rm{red}}}
\newcommand \Gal {{\rm {Gal}}}
\newcommand \PGL {{\rm {PGL}}}
\newcommand \Bl {{\rm {Bl}}}
\newcommand \Sing {{\rm {Sing}}}
\newcommand \spn {{\rm {span}}}
\newcommand \Nm {{\rm {Nm}}}
\newcommand \inv {{\rm {inv}}}
\newcommand \codim {{\rm {codim}}}
\newcommand \Div{{\rm{Div}}}
\newcommand \sg {{\Sigma }}
\newcommand \DM {{\sf DM}}
\newcommand \Gm {{{\mathbb G}_{\rm m}}}
\newcommand \tame {\rm {tame }}
\newcommand \znak {{\natural }}
\newcommand \lra {\longrightarrow}
\newcommand \hra {\hookrightarrow}
\newcommand \rra {\rightrightarrows}
\newcommand \ord {{\rm {ord}}}
\newcommand \Rat {{\mathscr Rat}}
\newcommand \rd {{\rm {red}}}
\newcommand \bSpec {{\bf {Spec}}}
\newcommand \Proj {{\rm {Proj}}}
\newcommand \pdiv {{\rm {div}}}
\newcommand \CH {{\it {CH}}}
\newcommand \wt {\widetilde }
\newcommand \ac {\acute }
\newcommand \ch {\check }
\newcommand \ol {\overline }
\newcommand \Th {\Theta}

\newenvironment{pf}{\par\noindent{\em Proof}.}{\hfill\framebox(6,6)
\par\medskip}

\newtheorem{theorem}[subsection]{Theorem}
\newtheorem{conjecture}[subsection]{Conjecture}
\newtheorem{proposition}[subsection]{Proposition}
\newtheorem{lemma}[subsection]{Lemma}
\newtheorem{remark}[subsection]{Remark}
\newtheorem{remarks}[subsection]{Remarks}
\newtheorem{definition}[subsection]{Definition}
\newtheorem{corollary}[subsection]{Corollary}
\newtheorem{example}[subsection]{Example}
\newtheorem{examples}[subsection]{examples}

\title{On the closed embedding of the product of theta divisors into  product of Jacobians  and Chow groups}
\author{Kalyan Banerjee}

\address{Indian Statistical Institute, Bangalore Center, Bangalore 560059}


\email{kalyanb$_{-}$vs@isibang.ac.in}

\footnotetext{Mathematics Classification Number: 14C25, 14D05, 14D20,
 14D21}

\begin{abstract}
In this article we generalize the injectivity of the push-forward homomorphism at the level of Chow groups, induced by the closed embedding of  $\Sym^m C$ into $\Sym^n C$ for $m\leq n$, where $C$ is a smooth projective curve, to symmetric powers of a smooth projective variety of higher dimension. We also prove the analog of this theorem for product of symmetric powers of smooth projective varieties. As an application we prove the injectivity of the push-forward homomorphism  at the level of Chow groups, induced by the closed embedding of self product of theta divisor into the self product of the Jacobian of a smooth projective curve.
\end{abstract}

\maketitle


\section{Introduction}
Let $C$ be a smooth projective curve and let $\Sym^n C$ denote the $n$-th symmetric power of $C$. In \cite{Collino} Collino proved that if we consider the embedding of $\Sym^m C$ into $\Sym^n C$, then it induces an injective push-forward homomorphism at the level of Chow groups. In this text the author is curious whether the same thing holds true for for higher dimensional smooth projective varieties. In \cite{BI}, we prove that the embedding of $\Sym^m C$ into $\Sym^n C$ induces an injective push-forward homomorphism when we consider the Higher Chow groups. Also an analog of Collino's theorem was proved in \cite{BI} for some open subschemes of $\Sym^m C$ mapping into an open subscheme of $\Sym^n C$.

The first theorem in this direction is the generalisation of Collino's theorem for higher dimensional smooth projective varieties and for products of symmetric powers of smooth projective varieties, where we work with Chow groups with $\QQ$-coefficients. Let us denote the Chow group with $\QQ$-coefficients, for a projective variety $X$ by $\CH_*(X)$.

\textit{The push-forward homomorphism $i_*$ from $\CH_*(\Sym^m X)$ to $\CH_*(\Sym^n X)$ is injective.}

\textit{The push-forward homomorphism $i_{m_1,m_2*}^{n_1,n_2}$ from $\CH_*(\Sym^{m_1} X\times \Sym^{m_2} Y)$ to $\CH_*(\Sym^{n_1} X\times \Sym^{n_2} Y)$ is injective.}

The method used to prove this theorems are almost the same as in \cite{Collino}, the second theorem stated above requires some  modification of the proof given in Collino's paper \cite{Collino}[corllary $1$ and theorem $1$]. Next we try to understand the kernel of the push-forward homomorphism induced by the natural regular morphism from $\Sym^k X\times \Sym^k X$ to $\Sym^{2k+n}X$, for some positive integer $n$. The elements in the kernel have some nice form, specially when we consider $k=1$ and   zero cycles on $X^2$.



We use the machinary derived from Collino's theorem, for product of Jacobians and the embedding of product of Theta divisors into the product of Jacobians, to deduce that this embedding induces injective push-forward homomorphism at the level of Chow groups.

\textit{Let $C$ be a smooth projective curve of genus $g$. Consider the natural morphism from $\Sym^{g} C\times \cdots\times \Sym^{g} C$ to $J(C)\times \cdots\times J(C)$. Let $\Th$ be the theta divisor embedded into $J(C)$.
Let $j$ denote the closed embedding of $\Th\times \cdots\times \Th$
into $J(C)\times\cdots\times J(C)$. Then $j_*$ from $\CH_k(\Th\times\cdots\times \Th)$ to $\CH_k(J(C)\times\cdots J(C))$ is injective for $k\geq 0$.}

The method used to prove this theorem is the fact that the $g$-th symmetric power of a curve $C$ is birational to $J(C)$ and $g-1$-th symmetric power of $C$ is birational to $\Th_C$, the theta divisor of $J(C)$. Also we make use of the fact that $\Sym^n C$ is a projective bundle over $J(C)$ for high enough $n$.

{\small \textbf{Acknowledgements:} The author wishes to thank Jaya Iyer and Manish Kumar for useful discussions relevant to the proof of some theorems present in this paper. The author also wishes to thank the ISF-UGC grant for funding this project and hospitality of Indian Statistical Institute, Bangalore Center for hosting this project.}

\section{Collino's theorem for higher dimensional varieties and products of symmetric powers}
\label{Collino}

Let $X$ be a smooth projective curve defined over  an algebraically closed field. Let $\Sym^n X$ denote the $n$-th symmetric power of $X$. Let us fix a point $p$ in $X$. Consider the closed embedding $i_{m,n}$ of $\Sym^m X$ to $\Sym^n X$, given by
$$[x_1,\cdots,x_m]\mapsto [x_1,\cdots,x_m,p,\cdots,p]$$
where $[x_1,\cdots,x_m]$ denote the unordered $m$-tuple of points in $\Sym^m X$. Then the push-forward homomorphism $i_{m,n*}$ from $\CH_*(\Sym^m X)$ to $\CH_*(\Sym^n X)$ is injective as proved in \cite[Theorem 1]{Collino}. In this section we prove that the same holds for an arbitrary smooth projective variety $X$. That is the push-forward homomorphism $i_{m,n*}$ from $\CH_*(\Sym^m X)$ to $\CH_*(\Sym^n X)$ is injective, where $i_{m,n}$ is defined as before.  To prove that we follow the approach by Collino in \cite{Collino},  the argument present here is a generalisation of the arguments in \cite{Collino}. It is straightforward for the first case when $X$ is a smooth projective variety but requires minor modifications when we want to prove it for product of symmetric powers, but we write it for our convenience.

Let $\Gamma$ be the correspondence given by
$$\pi_n\times \pi_m(\Gamma')$$
supported on $\Sym^m X\times \Sym^n X$ where $\Gamma'$ is the graph of the projection $pr_{n,m}$ from $X^n$ to $X^m$ and $\pi_{n}$ is the natural morphism from
$X^n$ to $\Sym^n X$. Let $g_*$ be the homomorphism induced by $\Gamma$ at the level of algebraic cycles.

First we prove the following lemma.
\begin{lemma}
\label{lemma1}
The homomorphism $g_{*}\circ i_{m,n*}$ at the level of the Chow groups with $\QQ$-coefficients, is induced by the cycle $(i_{m,n}\times id)^*\Gamma$ on $\Sym^m X\times \Sym^m X$.
\end{lemma}
\begin{proof}
Let's denote $i_{m,n*}$ as $i_*$.
We have
$$g_*i_*(Z)=pr_{\Sym^m X*}(i_*(Z)\times \Sym^m X.\Gamma)\;.$$
The above expression can be written as
$$pr_{\Sym^m X*}((i\times id)_*(Z\times \Sym^m X).\Gamma)\;.$$
By the projection formula the above is equal to
$$pr_{\Sym^m X*}\circ(i\times id)_*
((Z\times \Sym^m X). (i\times id)^*\Gamma)\;.$$
Since $pr_{\Sym^m X}\circ(i\times id)$
is the projection $pr_{\Sym^m X}$ we get that the above is equal to
$$pr_{\Sym^m X*}((Z\times \Sym^m X). (i\times id)^*\Gamma)\;.$$
Here the above two projections are taken respectively on $\Sym^n X\times \Sym^m X$ and on $\Sym^m X\times\Sym^m X$.
So we get that $g_*\circ i_*$ is induced by $(i\times id)^*\Gamma$.

\end{proof}

Now let us consider the closed embedding $\Sym^{m-1}X$ into $\Sym^m X$. Let $\rho$ be the embedding of the complement of $\Sym^{m-1}X$ in $\Sym^m X$. Then we have the following proposition.
\begin{proposition}
\label{prop1}
At the level of the Chow group with rational co-efficients we have

$$\rho^{*}\circ g_*\circ i_*=\rho^{*}\;.$$
\end{proposition}
\begin{proof}
To prove the proposition we prove that
$$(i\times \id)^{-1}\Gamma=\Delta \cup D$$
where $\Delta $ means the diagonal in $\Sym^{m}X\times \Sym^{m}X$ and $D$ is a closed subscheme of $\Sym^{m}X\times \Sym^{m-1}X$.
For that we write out
$$(i\times \id)^{-1}\Gamma\;,$$
that is equal to
$$(i\times \id)^{-1}(\pi_n\times\pi_m)Graph(pr_{n,m})\;.$$
The above is equal to
$$(i\times \id)^{-1}(\pi_n\times\pi_m)
\{((x_1\cdots,x_n),(x_1,\cdots,x_m))|x_i\in X \}$$
that is
$$(i\times \id)^{-1}\{([x_1,\cdots,x_n],[x_1,\cdots,x_m])|x_i\in X\}\;.$$
Call the set
$$\{([x_1,\cdots,x_n],[x_1,\cdots,x_m])|x_i\in X,\}$$
as $B$, and the set
$$(i\times \id)^{-1}\{([x_1,\cdots,x_n],[x_1,\cdots,x_m])|x_i\in X\}\;.$$
as $A$. The set $A$ is of the form
$$\{([x'_1,\cdots,x'_m],[y'_1,\cdots,y'_m])|
([x_1',\cdots,x_m',p,\cdots,p],[y_1',\cdots,y_m'])\in B\}\;.$$
So the set $A$ can be written as the union of
$$\{([x_1'\cdots,x_m'],[x_1'\cdots,x_m'])|x_i'\in X,\}$$
and
$$\{([x_1'\cdots,x_m'],[x_1'\cdots,p,x_m'])|x_i'\in X\}\;,$$
that is the union
$$\Delta\cup D$$
where $\Delta $ is the diagonal in the scheme $\Sym^{m}X\times \Sym^{m}X$ and $D$ is a closed subscheme in $\Sym^m X\times \Sym^{m-1}X\;.$
Therefore arguing as in \cite{Collino} [proposition 1]we get that
$$(i\times id)^*(\Gamma)=\Delta+Y$$
as an algebraic cycle,
where $Y$ is supported on $\Sym^m X\times \Sym^{m-1}X$.
So $g_*i_*(Z)$ is equal to
$$\pr_{\Sym^m X*}[(\Delta+Y).(Z\times \Sym^m X)]=Z+Z_1$$
where $Z_1$ is supported on $\Sym^{m-1}X$. For the above we have to take some care about defining the intersection product. This Can be done since $X$ is smooth and we have the identification $\CH_*(X/G)=\CH_*(X)^G$ \cite{Fulton}[Example 1.7.6]. So
$$\rho^{*}g_*i_*=\rho^{*}(Z+Z_1)=\rho^{*}(Z)$$
since $\rho^{*}(Z_1)=0$. Hence the proposition is proved.
\end{proof}

Now we prove that the push-forward homomorphism $i_*$ from $\CH_*(\Sym^m X)$ to $CH_*(\Sym^n X)$ is injective.

\begin{theorem}
\label{theorem1}
The push-forward homomorphism $i_*$ from $\CH^*(\Sym^m X)$ to $\CH^*(\Sym^n X)$ is injective.
\end{theorem}
\begin{proof}
We prove this by induction. First $\Sym^0 X$ is a single point and the morphism $i_{0,n}=(p,\cdots,p)$. To show that $i_{0,n}$ is injective we use the definition of the rational equivalence. First we show that the inclusion of $\Sym^0 X$ into $X$ gives an injection at the level of Chow groups. Suppose that some multiple of $p$ is rationally equivalent to zero on $X$. Then there exists a irreducible curve $C$ inside $X$ and a non-zero rational function $f$ in $k(C)$ such that
 $$
 np=div(f)\;.
 $$
But by the Collino's theorem for smooth projective curves it follows that $n=0$ and hence we have, the push-forward induced by this morphism $i_{0,n}$ is injective. Assume now that $i_*$ is injective for $m-1$ and any $n$ greater than or equal to $m-1$. Then consider the following commutative diagram

$$
  \xymatrix{
   \CH^*(\Sym^{m-1} X) \ar[r]^-{i_{m-1,m*}} \ar[dd]_-{}
  &   \CH^*(\Sym^m X) \ar[r]^-{\rho^{*}} \ar[dd]_-{i_{mn*}}
  & \CH^*(X_0(m))  \ar[dd]_-{}  \
  \\ \\
  \CH^*(\Sym^{m-1} X) \ar[r]^-{i_{m-1,n*}}
    & \CH^*(\Sym^n X) \ar[r]^-{}
  & \CH^*({(\Sym^{m-1}X)^c})
  }
$$
In the above $(\Sym^{m-1}X)^c$ is the complement of $\Sym^{m-1}X$ in $\Sym^n X$.
In this diagram the left part of the two rows are exact by the induction hypothesis and the middle part is exact by the localization exact sequence for  Chow groups.
Now suppose that $z$ belongs to $\CH^*(\Sym^m X)$, such that $$i_{m,n*}(z)=0$$
and let $Z$ be the cycle such that the cycle class of $Z$ is $z$. Let $cl(Z)$ denote the cycle class in the  Chow group, corresponding to the algebraic cycle $Z$.

Then we have
$$cl(\rho^{*}g_*i_*(Z))=0$$
which means by the theorem \ref{theorem1}
$$cl(\rho^{*}(Z))=0\;,$$
hence
$$\rho^{*}(cl(Z))=\rho^{*}(z)=0\;.$$
$$$$
So by the localization exact sequence there exists $z'$ in $\CH^*(\Sym^{m-1}X)$, such that
$$z=i_{m-1,m*}(z')\;.$$
By the commutativity of the left square of the above commutative diagram we get that
$$i_{m-1,n*}(z')=0\;.$$
By the injectivity of $i_{m-1,n*}$ we get that $z'=0$, so $z=0$, hence $i_{m,n*}$ is injective.
\end{proof}

\subsection{Collino's theorem on products of symmetric powers}
\label{subsection2}
Let $X,Y$ be  smooth projective varieties over an algebraically closed field $k$. Then let us fix two points $p,q$ on $X,Y$ respectively. Consider the map
$$([x_1,\cdots,x_{m_1}],[y_1,\cdots,y_{m_2}])\mapsto ([x_1,\cdots,x_{m_1},p,\cdots,p],[y_1,\cdots,y_{m_2},q,\cdots,q])$$
from $\Sym^{m_1} X\times \Sym^{m_2}Y$ to $\Sym^{n_1} X\times \Sym^{n_2} Y$. Denote this map as $i_{m_1,m_2}^{n_1,n_2}$ as previous. Then we prove that $i_{m_1,m_2*}^{n_1,n_2}$ is injective from $\CH_k(\Sym^{m_1} X\times \Sym^{m_2} Y)$ to $\CH_k(\Sym^{n_1} X\times \Sym^{n_2} Y)$. The arguments do not differ much from the previous arguments, only difficult part is to prove the analog of \ref{prop1}. First we define the correspondence $\Gamma$ to be
$$
(\pi_{n_1}^X\times \pi_{n_2}^{Y})\times(\pi_{m_1}^X\times \pi_{m_2}^{Y})(\Gamma')
$$
where $\Gamma'$ is the graph of the projection $pr$ from $X^{n_1}\times Y^{n_2}$ to $X^{m_1}\times Y^{m_2}$. Consider the homomorphism $g$ induced by $\Gamma$. Then we have the following lemma which is the analog of \ref{lemma1}.
\begin{lemma}
\label{lemma2}
The homomorphism $g_{*}\circ i_{m_1,m_2*}^{n_1,n_2}$ at the level of the Chow group with rational co-efficients, is induced by the cycle $(i_{m_1,m_2}^{n_1,n_2}\times id)^*\Gamma$ on $(\Sym^{m_1} X\times \Sym^{m_2} Y)\times(\Sym^{m_1} X\times \Sym^{m_2} Y)$.
\end{lemma}
\begin{proof}
The proof is same as in the proof of lemma \ref{lemma1}, replacing $\Sym^m X$ by $\Sym^{m_1} X\times \Sym^{m_2} Y$.
\end{proof}
Let $X_0(m)\times Y_0(m)$ be the complement of
$$\Sym^{m_1-1}X\times \Sym^{m_2}Y\cup \Sym^{m_1} X\times \Sym^{m_2-1}Y$$ inside $\Sym^{m_1} X\times \Sym^{m_2} Y$. Let $\rho$ denote the embedding of $X_0(m)\times Y_0(m)$. Then we have the following.
\begin{proposition}
\label{prop2}
Let $\rho^*$ be the pull-back homomorphism from $\CH_*(\Sym^{m_1} X\times \Sym^{m_2} Y)$ to $\CH_*(X_0(m)\times Y_0(m))$. Then we have
$$\rho^*\circ g_*\circ i_*=\rho^*\;.$$
Here $i$ denotes $i_{m_1,m_2}^{n_1,n_2}$.
\end{proposition}

\begin{proof}
First we prove as in \ref{prop1}
$$(i\times \id)^{-1}(\Gamma)=\Delta\cup D$$
where $\Delta$ is the diagonal in $(\Sym^{m_1} X\times \Sym^{m_2} Y)\times (\Sym^{m_1} X\times \Sym^{m_2} Y)$. $D$ is supported on
$$(\Sym^{m_1-1}X\times \Sym^{m_2} Y)\cup (\Sym^{m_1} X\times \Sym^{m_2-1}Y)\;.$$
$$(i\times id)^{-1}(\Gamma)=\{(([x_1,\cdots,x_{n_1}],[y_1,\cdots,y_{n_2}]),
([x'_1,\cdots,x'_{m_1}],[y'_1,\cdots,y'_{m_2}])$$
$$|x_i,x_i'\in X,y_i,y_i'\in Y\}$$
Also we have
$$x_i'=\sigma(x_i)\quad y_i'=\tau(y_i)$$
for some $\sigma,\tau$ in the group of permutations of $n$-letters.
Call the support of $\Gamma$ to be
as $B$ and $(i\times id)^{-1}(\Gamma)=A$. Then $A$ consists of pairs of pairs of unordered tuples like
$$(([x_1,\cdots,x_{m_1}],[y_1,\cdots,y_{m_2}]),$$
$$([x_1',\cdots,x_{m_1}'],[y_1',\cdots,y_{m_2}']))
$$
such that
$$
(([x_1,\cdots,x_{m_1},p,\cdots,p],[y_1,\cdots,y_{m_2},\cdots,q]),
([x_1',\cdots,x_{m_1}'],[y_1',\cdots,y_{m_2}']))
$$
is in $B$. So the elements of $A$ are either of the form
$$(([x_1,\cdots,x_{m_1}],[y_1,\cdots,y_{m_2}])
([x_1,\cdots,x_{m_1}],[y_1,\cdots,y_{m_2}]))$$
or of the form
$$(([x_1,\cdots,x_{m_1}]),[y_1,\cdots,y_{m_2}]),([p,\cdots,x_{m_1}],
[y_1,\cdots,y_{m_2}]))$$
or of the form
$$(([x_1,\cdots,x_{m_1}],[y_1,\cdots,y_{m_2}]),([x_1,\cdots,x_{m_1}],
[q,\cdots,y_{m_2}]))\;.$$
So we can write
$$(i\times id)^{-1}(\Gamma)=\Delta\cup D$$
where $\Delta$ is the diagonal of $(\Sym^{m_1} X\times \Sym^{m_2} Y)\times(\Sym^{m_1} X\times \Sym^{m_2} Y)$ and $D$ is supported on
$$(\Sym^{m_1} X\times \Sym^{m_1} Y)\times (\Sym^{m_1-1}X\times \Sym^{m_2} Y)$$
$$\cup (\Sym^{m_1} X\times \Sym^{m_2} Y)\times (\Sym^{m_1}X\times \Sym^{m_2-1} Y)\;.$$
Then we have $g_*i_*$ is equal to
$$\pr_{(\Sym^{m_1} X\times \Sym^{m_2} Y*)}((\Delta+Y_1).(Z\times Sym^{m_1} X\times \Sym^{m_2} Y))$$
where $Y_1$ has support $D$ (it follows that the cycle $(i\times id)^*\Gamma=\Delta+Y_1$ as cycles by arguing as in proposition $1$ in \cite{Collino}). Here again we have to take care about the Chow moving lemma, which is true on $\Sym^{m_1}X\times \Sym^{m_2}Y$, because the variety is a quotient of $X^{m_1}\times Y^{m_2}$ by the group $S_{m_1}\times S_{m_2}$, where $S_i$ denote the symmetric group on $i$-letters. So the above can be written as
$$Z+Z_1$$
where the support of $Z_1$ is contained in
$$(\Sym^{m_1-1}X\times \Sym^{m_2} Y)\cup(\Sym^{m_1}X\times \Sym^{m_2-1}Y)\;.$$
Therefore
$$\rho^*(Z+Z_1)=\rho^*(Z)$$
since $\rho^*(Z_1)$ is zero. So we have the proposition.

\end{proof}

Now we prove that the push-forward homomorphism $i_{m_1,m_2*}^{n_1,n_2}$ is injective from $\CH_k(\Sym^{m_1} X\times \Sym^{m_2} Y)$ to $\CH_k(\Sym^{n_1} X\times \Sym^{n_2} Y)$. Denote the closed embedding of $\Sym^{m_1-1}X\times \Sym^{m_2}Y\cup \Sym^{m_1}X\times \Sym^{m_2-1}Y$ into $\Sym^{m_1}X\times \Sym^{m_2}Y$  as $j$ and that into $\Sym^{n_1}X\times \Sym^{n_2}Y$ as $j'$.
\begin{theorem}
\label{theorem3}
The push-forward homomorphism $i_{m_1,m_2*}^{n_1,n_2}$ from $\CH_*(\Sym^{m_1} X\times \Sym^{m_2} Y)$ to $\CH_*(\Sym^{n_1} X\times \Sym^{n_2} Y)$ is injective.
\end{theorem}
\begin{proof}
The proof follows by mimicking the arguments of \ref{theorem1} with $\Sym^m X$ is replaced by $\Sym^{m_1} X\times \Sym^{m_2} Y$ and by using \ref{prop2}. The case when $m=0$ follows, arguing similarly as in \ref{theorem1}. Let us assume that $i_{k_1,k_2*}^{n_1,n_2}$ is injective for $k_1=0,\cdots,m_1-1$ or $k_2=0,\cdots,m_2-1$.

Here $g_*$ is the homomorphism $\Gamma$ defined by the correspondence $\Gamma$ mentioned in the beginning of this subsection. So let us consider the following commutative diagram. Let $A$ be the union of $\Sym^{m_1-1}X\times \Sym^{m_2}Y$ and $\Sym^{m_1}X\times \Sym^{m_2-1}Y$
$$
  \xymatrix{
   \CH^*(A) \ar[r]^-{j_{*}} \ar[dd]_-{}
  &   \CH^*(\Sym^{m_1} X\times \Sym^{m_2} Y) \ar[r]^-{\rho_0^{*}} \ar[dd]_-{i_{mn*}}
  & \CH^*(X_0(m_1)\times Y_0(m_2))  \ar[dd]_-{}  \
  \\ \\
   \CH^*(A) \ar[r]^-{j'_*}
    & \CH^*(\Sym^{n_1} X\times \Sym^{n_2} Y) \ar[r]^-{}
  & \CH^*(V)
  }
$$
Here $V$ is the complement of $\Sym^{m_1-1}X\times \Sym^{m_2}Y\cup \Sym^{m_1}X\times \Sym^{m_2-1}Y$ in $\Sym^{n_1}X\times \Sym^{n_2} Y$. The map $\rho$ denote the inclusion of $U$ into $\Sym^{m_1} X\times \Sym^{m_2} Y$.
Let $z$ belongs to kernel of $i_{m_1,m_2*}^{n_1,n_2}$, denote it by $i$ that is
$$i_{*}(z)=0\;.$$
Then by composing with $\rho^*$ coming from \ref{prop2} and $g_*$ we get that
$$\rho^*g_*(i_{*}(z))=0$$
but the above is nothing but
$$\rho^*(z)=0$$
by \ref{prop2}.

Therefore by the localisation exact sequence present in first row of the previous diagram we get that there exists $z'$ in $\CH^*(\Sym^{m_1-1}X\times \Sym^{m_2}Y\cup \Sym^{m_1}X\times \Sym^{m_2-1}Y)$ such that
$$z=j_{*}(z')$$
by the commutativity of the previous rectangle it follows that
$$j'_{*}(z')=0\;.$$
Now note that there is an exact sequence of Chow groups as follows.
Let $A=A_1\cup A_2$ be the union of irreducible components of $A$, then we have
$$\CH_*(A_1\cap A_2)\to \CH_*(A_1)\oplus \CH_*(A_2)\to \CH_*(A)\to 0\;.$$
Applying this in our situation when $A_1=\Sym^{m_1-1}X\times \Sym^{m_2}Y$ and $A_2=\Sym^{m_1}X\times \Sym^{m_2-1}Y$ and let $A$ be their union, we get that there exist $z''$ such that $f(z'')=z'$, where $f$ is the homomorphism from $\CH_*(A_1)\oplus \CH_*(A_2)$ to $\CH_*(A)$, composing this with $j'_{*}$ we get that
$$j'_*(f(z''))=0$$
but $j'_*f$ is nothing but the homomorphism
$$\CH_*(A_1)\oplus \CH_*(A_2)\to \CH_*(\Sym^{n_1}X\times \Sym^{n_2}Y)$$
we prove that $z''$ is actually in the kernel of the above homomorphism and the kernel is $\CH_*(A_1\cap A_2)$, therefore it is zero by the induction hypothesis because $A_1\cap A_2=\Sym^{m_1-1}X\times \Sym^{m_2-1}Y$ . This is done by showing that $\CH_*(A_1\cup A_2)$ to $\CH_*(\Sym^{n_1}X\times \Sym^{n_2}Y)$ is injective. This can be achieved by arguing similarly as above using the technique of producing a correspondence on an appropriate variety and considering the push-forward induced by that. The Chow moving lemma is taken care of because we are working with varieties which are union of quotients of a smooth projective variety by a finite group.

\end{proof}
\subsection{Kernel of the push-forward homomorphism from  $\CH_*(\Sym^k X\times \Sym^k X)$ to $\CH_*(\Sym^{2k+n}X)$}
Let $p$ be a fixed point on a smooth projective variety $X$. Consider the morphism $i$ from $\Sym^k X\times \Sym^k X$ to $\Sym^{2k+n}X$ given by $$([x_1,\cdots,x_k],[y_1,\cdots,y_k])\mapsto [x_1,\cdots,x_k,y_1,\cdots,y_k,p,\cdots,p]\;.$$
We want to prove that the push-forward homomorphism induced by this regular morphism is injective at the level of Chow groups. Let us consider the projection morphism $pr$ from $X^{2k+n}$ to $X^{2k}\cong X^k\times X^k$. Let $\pi_i$ be the quotient morphism from $X^i$ to $\Sym^i X$. Then consider the correspondence
$$\Gamma=\pi_{2k+n}\times \pi_k\times \pi_k(Graph(pr))$$
which is supported on $\Sym^{2k+n}X\times \Sym^k X\times \Sym^k X$.
Let $g_*$ denote the homomorphism at the level of Chow groups induced by the correspondence $\Gamma$.

\begin{lemma}
\label{lemma3}
The homomorphism $g_{*}\circ i_{*}$ at the level of Chow group with rational coefficients, is induced by the cycle $(i\times id)^*\Gamma$ on
$\Sym^k X\times \Sym^k X\times \Sym^k X\times \Sym^k X$.
\end{lemma}
\begin{proof}
The proof is same as in the proof of lemma \ref{lemma1}, replacing $\Sym^m X$ by $\Sym^k X\times \Sym^k X$ and by using the projection formula.
\end{proof}

Let $X_0(k)$ be the complement of
$$\Sym^{k-1}X\times \Sym^{k}X\cup \Sym^{k} X\times \Sym^{k-1}X$$ inside $\Sym^{k} X\times \Sym^{k} X$. Let $\rho$ denote the embedding of $X_0(k)$ into $\Sym^k X\times \Sym^k X$. Then we have the following.
\begin{proposition}
\label{prop3}
Let $\rho^*$ be the pull-back homomorphism from $\CH_*(\Sym^{k} X\times \Sym^{k} X)$ to $\CH_*(X_0(k))$. Then we have
$$\rho^*\circ g_*\circ i_*(Z)=\rho^*(dZ+\sum_i d_iZ_i)\;,$$
where $Z_i$ is the algebraic subset
$$\{([y_1,\cdots,y_i,x_{i+1},\cdots,x_k],
[x_1,\cdots,x_i,y_{i+1},\cdots,y_k]:([x_1,\cdots,x_k],[y_1,\cdots,y_k])\in Z\}\;.$$

\end{proposition}

\begin{proof}
First we prove as in \ref{prop1}
$$(i\times \id)^{-1}(\Gamma)=\Delta\cup \Delta_i\cup D$$
where $\Delta$ is the diagonal in $(\Sym^{k} X\times \Sym^{k} X)\times (\Sym^{k} X\times \Sym^{k} X)$. $\Delta_i$ is the subset of the form
$$\{([x_1,\cdots,x_k],[y_1,\cdots,y_k]),([y_1,\cdots,y_i,x_{i+1},\cdots,x_k]
,[x_1,\cdots,x_i,y_{i+1},\cdots,y_k])\}$$
and $D$ is supported on
$$(\Sym^{k-1}X\times \Sym^{k} X)\cup (\Sym^{k} X\times \Sym^{k-1}X)\;.$$

Call the support of $\Gamma$ to be
as $B$ and $(i\times id)^{-1}(\Gamma)=A$. Then $A$ consists of pairs of pairs of unordered tuples like
$$(([x_1,\cdots,x_{k}],[y_1,\cdots,y_{k}]),([x_1',\cdots,x_{k}'],
[y_1',\cdots,y_{k}']))
$$
such that
$$
(([x_1,\cdots,x_{k},y_1,\cdots,y_{k},p,\cdots,p]),
([x_1',\cdots,x_{k}'],[y_1',\cdots,y_{k}']))
$$
is in $B$. So the elements of $A$ are either of the form
$$(([x_1,\cdots,x_{k}],[y_1,\cdots,y_{k}])
([x_1,\cdots,x_{k}],[y_1,\cdots,y_{k}]))$$
or of the form
$$(([x_1,\cdots,x_{k}]),[y_1,\cdots,y_{k}]),([p,\cdots,x_{k}],
[y_1,\cdots,y_{k}]))$$
or of the form
$$(([x_1,\cdots,x_{k}],[y_1,\cdots,y_{k}]),([x_1,\cdots,x_{k}],
[p,\cdots,y_{k}]))\;,$$
or of the form
$$(([x_1,\cdots,x_{k}],[y_1,\cdots,y_{k}]),([y_1,\cdots,y_i,\cdots,x_{k}],
[x_1,\cdots,x_i,\cdots,y_{k}]))\;,$$

So we can write
$$(i\times id)^{-1}(\Gamma)=\Delta\cup_i \Delta_i \cup D$$
where $\Delta$ is the diagonal of $(\Sym^{k} X\times \Sym^{k} X)\times(\Sym^{k} X\times \Sym^{k} X)$ and $D$ is supported on
$$(\Sym^{k} X\times \Sym^{k} X)\times (\Sym^{k-1}X\times \Sym^{k} X)\cup (\Sym^{k} X\times \Sym^{k} X)\times (\Sym^{k}X\times \Sym^{k-1} X)\;,$$
and $\Delta_i$'s are described as in the beginning. So the cycle
$$(i\times\id)^*(\Gamma)=d\Delta+\sum_i d_i\Delta_i+Y$$ where $Supp(Y)=D$.
Then we have $g_*i_*$ is equal to
$$\pr_{(\Sym^{k} X\times \Sym^{k} X*)}((d\Delta+\sum_i d_i \Delta_i+Y).(Z\times Sym^{k} X\times \Sym^{k} X))$$
where $Y$ has support $D$. So the above can be written as
$$dZ+\sum_i d_i Z_i + Z'$$
where the support of $Z'$ is contained in
$$(\Sym^{k-1}X\times \Sym^{k} X)\cup(\Sym^{k}X\times \Sym^{k-1}X)\;.$$
Therefore
$$\rho^*(dZ+\sum_i d_iZ_i+Z')=d\rho^*(Z)+\sum_i d_i\rho^*(Z_i)$$
since $\rho^*(Z')$ is zero, where $Z_i$'s are as described in the statement of the proposition. So we have the proposition.

\end{proof}

\begin{theorem}
\label{theorem4}
Let $Z$ be a cycle belonging to the kernel of $i_*$. Suppose that the natural push-forward homomorphism from $\CH_*(\Sym^{k-1}X\times \Sym^k X\cup\Sym^k X\times \Sym^{k-1}X)$ to $\CH_*(\Sym^{2k+n}X)$. Then $dZ$ is rationally equivalent to the cycle $-\sum_i d_iZ_i$, where $Z_i$'s are as described as in the previous proposition \ref{prop3}.
\end{theorem}
\begin{proof}

Here $g_*$ is the homomorphism $\Gamma$ defined by the correspondence $\Gamma$ mentioned in the beginning of this subsection. So let us consider the following commutative diagram. Let $A$ be the union of $\Sym^{k-1}X\times \Sym^{k}X$ and $\Sym^{k}X\times \Sym^{k-1}X$
$$
  \xymatrix{
   \CH^*(A) \ar[r]^-{j_{*}} \ar[dd]_-{}
  &   \CH^*(\Sym^{k} X\times \Sym^{k} X) \ar[r]^-{\rho^{*}} \ar[dd]_-{i_{*}}
  & \CH^*(X_0(k))  \ar[dd]_-{}  \
  \\ \\
 \CH^*(A) \ar[r]^-{j'_*}
    & \CH^*(\Sym^{2k+n} X) \ar[r]^-{}
  & \CH^*(V)
  }
$$
Here $V$ is the complement of $\Sym^{k-1}X\times \Sym^{k}X\cup \Sym^{k}X\times \Sym^{k-1}X$ in $\Sym^{k}X\times \Sym^{k} X$. The map $\rho$ denote the inclusion of $U$ into $\Sym^{k} X\times \Sym^{k} X$.
Let $Z$ belongs to kernel of $i_{*}$, that is
$$i_{*}(Z)=0\;.$$
Then by composing with $\rho^*$ coming from \ref{prop2} and $g_*$ we get that
$$\rho^*g_*(i_{*}(Z))=0$$
but the above is nothing but
$$\rho^*(dZ+\sum_i d_iZ_i)=0$$
by \ref{prop3}.

Therefore by the localisation exact sequence present in first row of the previous diagram we get that there exists $Z'$ in $\CH_*(\Sym^{k-1}X\times \Sym^{k}X\cup \Sym^{k}X\times \Sym^{k-1}X)$ such that
$$dZ+\sum_i d_iZ_i=j_{*}(Z')$$
by the commutativity of the previous rectangle it follows that
$$j'_{*}(Z')=0\;.$$
By the assumption it will follow that $Z'$ is rationally equivalent to zero, hence $dZ+\sum_i d_iZ_i$ is rationally equivalent to zero. Therefore we have the required result that $dZ$ is rationally equivalent to $-\sum_i d_iZ_i$.

\end{proof}

\begin{example}
Now consider $k=1$, then we have the push-forward homomorphism from $\CH_*(X^2)$ to $\CH_*(\Sym^{n+2}X)$. By theorem \ref{theorem1} we have that the push-forward homomorphism $\CH_*(X)$ to $\CH_*(\Sym^{n+2}X)$ is injective. Therefore if we take $Z$ to be an algebraic cycle in the kernel in the push-forward homomorphism $i_*$, we get that $Z$ is rationally equivalent to some rational multiple $-Z^t$ if we tensor the Chow groups with $\mathbb {Q}$, where $Supp(Z^t)=\{(y,x)|(x,y)\in Supp Z\}$. In particular let $z$ be a zero cycle in the kernel of $i_*$. Write
$$z=\sum_i(x_i,y_i)$$
then $z$ is rationally equivalent to
$$-\frac{m}{n}z^t=-\frac{m}{n}\sum_i(y_i,x_i)$$
since $z$ belongs to the kernel of $i_*$, $-\frac{m}{n}z^t$ also belongs to the kernel of $i_*$.

\end{example}

\section{Injectivity of the kernel of the push-forward homomorphism: some examples}
In this section we derive some nice consequences of the Collino's theorem. 
\begin{theorem}
\label{theorem5}
Let $C$ be a curve of  genus $g$. Consider the natural morphism from $\Sym^g C\times \Sym^{g} C$ to $J(C)\times J(C)$. Let $\Th$ denote the theta divisor in $J(C)$. Let $j$ denote the closed embedding of $\Th\times \Th$ into $J(C)\times J(C)$. Then the push-forward homomorphism $j_*$ from $\CH_k(\Th\times\Th)$ to $\CH_k(J(C)\times J(C))$ is injective.
\end{theorem}
\begin{proof}
First of all notice that. 

$$
  \diagram
  \Sym^{g-1}C\times \Sym^{g-1}C \ar[dd]_-{q_{\Th}} \ar[rr]^-{j'} & & \Sym^g C\times \Sym^{g}C \ar[dd]^-{q} \\ \\
  \Th\times \Th \ar[rr]^-{j} & & J(C)\times J(C)
  \enddiagram
  $$
Considering the above commutative diagram, the proof goes in the same line as in \ref{subsection2}, where we consider a correspondence $\Gamma_1$, and then we consider the push-forward induced by $\Gamma_1$. Here  $\Gamma_1$ will be $(q\times q_{\Th})(\Gamma)$, where $\Gamma$ is as in \ref{subsection2}. Then we prove that $\rho^*\Gamma_*j_*$ is equal to $\rho^*$, where $\rho$ is the open embedding of $q((\Sym^{g-2}C\times \Sym^{g-1} C)\cup (\Sym^{g-1} C\times \Sym^{g-2}C))$. Then we proceed as in \ref{subsection2}.

\end{proof}
\begin{corollary}
\label{corollary2}
Let $C$ be a smooth projective curve of genus $g$. Consider the natural morphism from $\Sym^{g} C\times \cdots\times \Sym^{g} C$ to $J(C)\times \cdots\times J(C)$. Let $\Th$ be the  theta divisor embedded into $J(C)$.
Let $j$ denote the closed embedding of $\Th\times \cdots\times \Th$
into $J(C)\times\cdots\times J(C)$. Then $j_*$ from $\CH_k(\Th\times\cdots\times \Th)$ to $\CH_k(J(C)\times\cdots J(C))$ is injective for $k\geq 0$.
\end{corollary}
\begin{proof}
The proof follows from the previous theorem \ref{theorem5}.

\end{proof}

Let $i$ be involution on an abelian variety $A$ given by
$$i(a)=-a\;.$$
\begin{corollary}
Let $C$ be as in the previous corollary \ref{corollary2}, and suppose that $i(\Th)=\Th$. Then the push-forward homomorphism induced by the closed embedding of $\Th\times\cdots\times \Th/\{i\}$ to $J(C)\times\cdots\times J(C)/\{i\}$ is injective at the level of $\CH_k$.
\end{corollary}
\begin{proof}
Follows from corollary \ref{corollary2} and from the fact that for a projective variety $X$ and a finite group $G$ acting on $X$ we have
$$\CH_k(X/G)=\CH_k(X)^G$$
from \cite{Fulton}, example $1.7.6$.
where $\CH_k(X)^G$ denote the $G$-invariants of $\CH_k(X)$ under the action of $G$.
\end{proof}


\end{document}